\documentclass[11pt]{article}
\usepackage{indentfirst}
\usepackage{mathrsfs}
\usepackage{latexsym,lineno}
\usepackage{algorithm}
\usepackage{enumerate}
\usepackage{algorithmic}
\usepackage{epsfig}
\usepackage{color}
\usepackage{amsmath}\usepackage{fleqn}\usepackage{verbatim}
\usepackage{epsf}
\usepackage{amsthm}\usepackage{graphicx, float}\usepackage{graphicx}
\usepackage{amsfonts}
\usepackage{amssymb}\usepackage{graphpap}
\usepackage{epic}\usepackage{curves}
\usepackage{tikz}
\usepackage{subfigure}
\usepackage{mathrsfs}
\usepackage{pstricks}
\usepackage{hyperref}

\newtheorem{theorem}{Theorem}[section]

\newtheorem{corollary}{Corollary}
\newtheorem{lemma}{Lemma}[section]
\newtheorem{conjecture}[theorem]{Conjecture}
\numberwithin{equation}{section}

\title{\bf Extremal graphs without long paths and \\a given graph\thanks {Research was partially supported by the National
Nature Science Foundation of China (grant number
12331012}}
\date{}
\author {Yichong Liu, \,  Liying Kang\thanks{\em Corresponding author. Email address: lykang@shu.edu.cn (L. Kang), lyc328az@163.com (Y. Liu)}\\
{\small Department of Mathematics, Shanghai University,
Shanghai 200444, P.R. China}}

\begin{document}

\maketitle

\begin{abstract}
   For a family of graphs $\mathcal{F}$, the Tur\'{a}n number $ex(n,\mathcal{F})$ is the maximum number of edges in an $n$-vertex graph containing no member of $\mathcal{F}$ as a subgraph. The maximum number of edges in an $n$-vertex  connected graph containing no member of $\mathcal{F}$ as a subgraph is denoted by $ex_{conn}(n,\mathcal{F})$. Let $P_k$ be the path on $k$ vertices and $H$ be a graph with chromatic number more than $2$. Katona and Xiao [Extremal graphs without long paths and large cliques, European J. Combin., 2023 103807] posed the following conjecture: Suppose that the chromatic number of $H$ is more than $2$. Then $ex\big(n,\{H,P_k\}\big)=n\max\big\{\big\lfloor \frac{k}{2}\big\rfloor-1,\frac{ex(k-1,H)}{k-1}\big\}+O_k(1)$.
   In this paper, we determine the exact value of $ex_{conn}\big(n,\{P_k,H\}\big)$ for sufficiently large $n$. Moreover, we obtain asymptotical result for $ex\big(n,\{P_k,H\}\big)$, which solves the conjecture proposed by Katona and Xiao.
   
   \bigskip

\medskip

\noindent{\bf AMS (2000) subject classification:}  05C35
\end{abstract}

\section{Introduction}
Let $G=(V,E)$ be a simple graph with vertex set $V=V(G)$ and edge set $E=E(G)$. The number of edges in $G$ is denoted by $e(G)$.
For $S\subseteq V(G)$, we denote by $N_G^{co}(S)$ the common neighborhood of vertices of $S$ in $V(G)$. Denote by $G[S]$ the graph induced by $S$, and denote by $G\setminus S$ the graph obtained from $G$ by deleting all vertices of $S$ and all edges incident with $S$. For $V_1,V_2\subseteq V(G)$, $E(V_1,V_2)$ denotes the set of edges between $V_1$ and $V_2$ in $G$, and $e(V_1,V_2)=|E(V_1,V_2)|$. The chromatic number of $G$ is denoted by $\chi(G)$.

Let $\mathcal{F}$ be a family of graphs. A graph $G$ is called $\mathcal{F}$-{\em free} if $G$ does not contain any member of the graphs in $\mathcal{F}$ as a subgraph. The {\em Tur\'{a}n number}, denoted by $ex(n,\mathcal{F})$, is the maximum number of edges in an $n$-vertex graph containing no member of $\mathcal{F}$ as a subgraph.
We call an $n$-vertex $\mathcal{F}$-free graph attaining $ex(n,\mathcal{F})$ edges an {\em extremal graph} for $\mathcal{F}$. The family of extremal graphs for $\mathcal{F}$ is denoted by $EX(n,\mathcal{F})$. The  maximum number of edges in an $n$-vertex  connected graph containing no member of $\mathcal{F}$ as a subgraph is denoted by $ex_{conn}(n,\mathcal{F})$, and the family of $n$-vertex $\mathcal{F}$-free connected graphs attaining $ex_{conn}(n,\mathcal{F})$ edges is denoted by $EX_{conn}(n,\mathcal{F})$.

Let $K_n$ denote the complete graph on $n$ vertices. Let $P_k$ and $C_k$ denote the path and the cycle on $k$ vertices respectively. The $n$-vertex independent set is denoted by $I_n$. The complete $k$-partite graph $K_{n_1,n_2,\cdots,n_k}$ is a graph formed by partitioning the set of $n$ vertices into $k$ subsets with $n_i(1\leq i\leq k)$ vertices in each subset, and connecting two vertices by an edge if and only if they belong to different subsets. The Tur\'{a}n graph $T(n, k)$ is a complete multipartite graph formed by partitioning the set of $n$ vertices into $k$ subsets with size as equal as possible. For two graphs $G$ and $H$, the disjoint union of $G$ and $H$ is denoted by $G\cup H$. The \textit{join} of $G$ and $H$, denoted by $G\vee H$, is the graph obtained from $G \cup H$ by adding all possible edges between $G$ and $H$.

In 1959, Erd{\H{o}}s and Gallai \cite{1959-erdos-p337} determined the Tur\'{a}n number for $P_k$.

\begin{theorem}[\cite{1959-erdos-p337}]\label{th_1}
   Let $G$ be an $n$-vertex graph. If $e(G)>\frac{k-2}{2}n$ where $k\geq 2$, then $G$ contains a copy of $P_k$.
\end{theorem}

The extremal problem for $P_k$ under the condition that $G$ is connected was considered by Kopylov \cite{1977-kopylov-maximal}.
He determined the Tur\'{a}n number for $P_k$. After 30 years,
Balister, Gy{\H{o}}ri, Lehel and Schelp \cite{2008-balister-p4487} found all the extremal graphs for $P_k$.

\begin{theorem}[\cite{2008-balister-p4487}]\label{th_2}
   Let $G$ be a connected graph on $n$ vertices containing no path on $k$ vertices where $n>k\geq 4$. Then
\begin{equation*}
   e(G)\leq \max\bigg\{\binom{k-2}{2}+(n-k+2),\binom{\big\lceil\frac{k}{2} \big\rceil}{2}+\left\lfloor \frac{k-2}{2}\right\rfloor\bigg(n-\Big\lceil \frac{k}{2}\Big\rceil\bigg)\bigg\}.
\end{equation*}
The equality holds when $G$ is either $(K_{k-3}\cup I_{n-k+2})\vee K_1$ or $\big(K_{k-2\lfloor\frac{k}{2} \rfloor+1}\cup I_{n-\lceil \frac{k}{2}\rceil}\big)\vee K_{\lfloor \frac{k}{2}\rfloor-1}$.
\end{theorem}

Recently, Katona and Xiao \cite{2023-Katona-withoutlongpath} determined the exact value of $ex(n,\{P_k,K_m\})$ if $k>2m-1$ and $ex_{conn}(n,\{P_k,K_m\})$ if $k>m$ for sufficiently large $n$.

\begin{theorem}[\cite{2023-Katona-withoutlongpath}]\label{th_3}
   Let $G$ be a connected $n$-vertex $\{K_m,P_k\}$-free graph where $m<k$. For sufficiently large $n\, (>N(k))$,
\begin{equation*}
ex_{conn}(n,\{K_m,P_k\})=\bigg(\Big\lfloor\frac{k}{2} \Big\rfloor-1\bigg)\bigg(n-\Big\lfloor\frac{k}{2} \Big\rfloor+1\bigg)+e\bigg(T\bigg(\Big\lfloor \frac{k}{2}\Big\rfloor-1, m-2\bigg)\bigg).
\end{equation*}
$T\big(\lfloor \frac{k}{2}\rfloor-1, m-2\big)\vee I_{n-\lfloor\frac{k}{2} \rfloor+1}$ is an  extremal graph.
\end{theorem}

\begin{theorem}[\cite{2023-Katona-withoutlongpath}]\label{th_4}
   Let $G$ be an $n$-vertex $\{K_m,P_k\}$-free graph where $2m-1<k$. For sufficiently large $n\, (>N'(k))$,
\begin{equation*}
ex(n,\{K_m,P_k\})=\bigg(\Big\lfloor\frac{k}{2} \Big\rfloor-1\bigg)\bigg(n-\Big\lfloor\frac{k}{2} \Big\rfloor+1\bigg)+e\bigg(T\bigg(\Big\lfloor \frac{k}{2}\Big\rfloor-1, m-2\bigg)\bigg).
\end{equation*}
$T\big(\lfloor \frac{k}{2}\rfloor-1, m-2\big)\vee I_{n-\lfloor\frac{k}{2} \rfloor+1}$ is an  extremal graph.
\end{theorem}

Katona and Xiao \cite{2023-Katona-withoutlongpath}  proposed to study $ex(n,\{P_k,H\})$ for an  $H$ with $\chi(H)>2$ and posed the following conjecture.

\begin{conjecture} [\cite{2023-Katona-withoutlongpath}]\label{conj_1}
   Suppose $\chi(H)>2$. Then
\begin{equation*}
ex(n,\{H,P_k\})=n\max\bigg\{\Big\lfloor \frac{k}{2}\Big\rfloor-1,\frac{ex(k-1,H)}{k-1}\bigg\}+O_k(1).
\end{equation*}
\end{conjecture}

In this paper, we determine the exact value of $ex_{conn}(n,\{P_k,H\})$ for sufficiently large $n$ and obtain asymptotical result of $ex(n,\{P_k,H\})$ which confirms Conjecture \ref{conj_1}.

Let $\mathcal{H}$ be the family of graphs obtained from $H$ by deleting a color class of $H$. Our main results are the following.

\begin{theorem}\label{theorem_1}
    If $H$ is a  graph with $\chi(H)> 2$ and $k$ is even, then
\begin{equation*}
ex_{conn}(n,\{P_k,H\})=ex\Big(\frac{k}{2}-1,\mathcal{H}\Big)+\Big(\frac{k}{2}-1\Big)\Big(n-\frac{k}{2}+1\Big)
\end{equation*}
 for sufficiently large $n$. The extremal graph is $T\vee I_{n-k/2+1}$, where $T\in EX(\frac{k}{2}-1,\mathcal{H})$.
\end{theorem}

\begin{theorem}\label{theorem_2}
   If $H$ is a  graph with $\chi(H)> 2$ and $k$ is  odd,  then
\begin{equation*}
ex_{conn}(n,\{P_k,H\})=ex\Big(\frac{k-3}{2},\mathcal{H}\Big)+\frac{k-3}{2}\Big(n-\frac{k-3}{2}\Big)+c
\end{equation*}
for sufficiently large $n$, where $c=0$ or $1$.
\end{theorem}

 The following theorem gives a sufficient condition for $c=0$ in Theorem \ref{theorem_2}.  Let $\mathcal{H}'$ be the family of graphs obtained from $H$ by deleting two adjacent vertices in $V(H)$.

\begin{theorem}\label{theorem_4}
    If $H$ is a graph with $\chi(H)>2$, $k$ is odd, and  every graph in $EX(\frac{k-3}{2},\mathcal{H})$ contains at least a member in $\mathcal{H}'$ , then
\begin{equation*}
ex_{conn}(n,\{P_k,H\})=ex\Big(\frac{k-3}{2},\mathcal{H}\Big)+\frac{k-3}{2}\Big(n-\frac{k-3}{2}\Big)
\end{equation*}
for sufficiently large $n$. The extremal graph is $T\vee I_{n-(k-3)/2}$, where $T$ is a graph in $EX(\frac{k-3}{2},\mathcal{H})$.
\end{theorem}

We will show that   Theorem \ref{th_3} can be  deduced from  Theorem \ref{theorem_4} in Section 3.

\begin{theorem}\label{theorem_3}
   Suppose $n$ is sufficiently large and $\chi(H)>2$. Then
\begin{equation*}
ex(n,\{H,P_k\})=n\max\bigg\{\Big\lfloor \frac{k}{2}\Big\rfloor-1,\frac{ex(k-1,H)}{k-1}\bigg\}+O_k(1).
\end{equation*}
\end{theorem}

The rest of the paper is organized as follows.
 In Section 2, we present some preliminaries and lemmas.
The proofs of main results will be given in  Sections 3 and 4. We give some discussion in the last section.

\section{Preliminaries}

   The following results can be found in \cite{2023-Katona-withoutlongpath}.
   \begin{lemma}[\cite{2023-Katona-withoutlongpath}] \label{lemma_1}
      Let $G$ be a connected graph on $n$ vertices with a path $P$ on $k-1$ vertices but no path on $k$ vertices. Let $u\in V(G)\setminus V(P)$ be a vertex adjacent to $s\geq 1$ vertices of $P$ and assume a longest path $Q$ in $G\setminus V(P)$ starting at $u$ has $j\, (j\geq 0)$ vertices. Then, $s+j\leq \lfloor \frac{k}{2}\rfloor$.
   \end{lemma}

   \begin{lemma}[\cite{2023-Katona-withoutlongpath}] \label{lemma_2}
      Let $G$ be a connected graph on $k$ vertices with no Hamiltonian path but with a path $P=(v_1,v_2,\cdots,v_{k-1})$ on $k-1$ vertices. Suppose the vertex $u\in V(G)\setminus V(P)$ is adjacent to $s$ vertices of $P$, that is $N_P(u)=\{v_{i_1},v_{i_2},\cdots,v_{i_s}\},i_1<i_2<\cdots<i_s$. Then \\
      (1) $s\leq \lfloor\frac{k}{2}\rfloor-1$;\\
      (2) there are no edges of the form $v_{i_j+1}v_{i_r+1},v_{i_j-1}v_{i_r-1},v_1v_{i_t+1}$ or $v_{i_t-1}v_{k-1}$, $1\leq j<r\leq s, 1\leq t\leq s$.
   \end{lemma}

   By Lemma \ref{lemma_2}, if $G$ is an $n$-vertex $P_k$-free connected graph with a path $P_{k-1}$, then $e(u,V(P_{k-1}))$ $\leq \lfloor\frac{k}{2}\rfloor-1$ for each $u\in V(G)\setminus V(P_{k-1})$. We  partition the vertices of $V(G)\setminus V(P_{k-1})$ into mutually disjoint sets in two different ways. \\
   (1) Let $u\in A_i$, if $e(u,V(P_{k-1}))=i,0\leq i\leq \lfloor\frac{k}{2}\rfloor-1$.\\
   (2) Let $u\in B_i$, if $u\in A_{i_1}$ and $u$ is connected to vertices in $A_{i_2},A_{i_3},\cdots, A_{i_r}$ by a path, then $i=\max\{i_1,i_2,\cdots,i_r\}$.

   One can easily verify that $A_i\cap A_j=\emptyset, B_i\cap B_j=\emptyset$ for $0\leq i\neq j\leq \lfloor\frac{k}{2}\rfloor-1$ and $$\sum_{i=0}^{\lfloor k/2\rfloor-1}A_i=\sum_{i=0}^{\lfloor k/2\rfloor-1}B_i=V(G)\setminus V(P_{k-1}).$$ Moreover, $B_0$ is empty since $G$ is connected, and there are no edges between $B_i$ and $B_j$ for $1\leq i\neq j\leq \lfloor\frac{k}{2}\rfloor-1$.

   \begin{lemma}[\cite{2023-Katona-withoutlongpath}] \label{lemma_3}
      Let $G$ be an $n$-vertex $P_k$-free connected graph with a path on $k-1$ vertices. Let $A_i$ and $B_i$ are the sets of vertices defined above. Then
\begin{equation*}
e(G)-e(G[V(P_{k-1})])\leq \Big(\Big\lfloor\frac{k}{2}\Big\rfloor-1\Big)\Big|A_{\lfloor k/2\rfloor-1}\Big|+\sum_{l=0}^{\lfloor k/2
\rfloor-2}\Big(\Big\lfloor\frac{k}{2}\Big\rfloor-\frac{3}{2}\Big)|A_l|.
\end{equation*}
\end{lemma}

   Now suppose $G$ is an $n$-vertex $\{P_k,H\}$-free connected graph, and $G$ contains a path on $k-1$ vertices. The following lemma estimates the number of  edges in $G[V(P_{k-1})]$.
   \begin{lemma} \label{lemma_4}
      For a  graph $H$ with $\chi(H)>2$ and sufficiently large $n$,
      let $G$ be an $n$-vertex $\{P_k,H\}$-free connected graph containing a $P_{k-1}$. Let $\mathcal{H}$ be the family of graphs obtained from $H$ by deleting a color class of $H$. If $\big|A_{\lfloor k/2\rfloor-1}\big|=\Theta(n)$, then
\begin{equation*}
e(G[V(P_{k-1})])\leq ex\Big(\Big\lfloor\frac{k}{2}\Big\rfloor-1,\mathcal{H}\Big)+\Big(\Big\lfloor\frac{k}{2}\Big\rfloor-1\Big)\Big\lceil \frac{k}{2}\Big\rceil+c,
\end{equation*}
where $c=0$ if $k$ is even and $c=1$ otherwise.
   \end{lemma}
   \begin{proof}
      Let $P_{k-1}:=(v_1,v_2,\dots,v_{k-1})$. We distinguish two different cases depending on the parity of $k$.

       {\em Case 1:} $k$ is even.
       Since for any vertex $u$ in $A_{k/2-1}$, $u$ cannot be adjacent to $v_1$, $v_{k-1}$, or consecutive vertices on $P_{k-1}$,  $N_{P_{k-1}}(u)=\{v_2,v_4,\cdots,v_{k-2}\}$. Then $A_{k/2-1}$ is  the common neighborhood of vertices in $\{v_2,v_4,\cdots,v_{k-2}\}$.
       The assumption $\big|A_{k/2-1}\big|=\Theta(n)$ implies that $G[\{v_2,v_4,\cdots,v_{k-2}\}]$ is $\mathcal{H}$-free. Otherwise, $G[\{v_2,v_4,\cdots,v_{k-2}\}]$
        contains a graph in $\mathcal{H}$, we can always find the missing color class in $A_{k/2-1}$ which forms an $H$, a contradiction. Thus $G[\{v_2,v_4,\cdots,v_{k-2}\}]$ must be $\mathcal{H}$-free. Moreover, by Lemma \ref{lemma_2}, there are no edges in $G[\{v_1,v_3,\cdots,v_{k-1}\}]$. Then $e(G[V(P_{k-1})])\leq ex\big(\frac{k}{2}-1,\mathcal{H}\big)+\big(\frac{k}{2}-1\big)\frac{k}{2}$.

       {\em Case 2:} $k$ is odd. Since for any vertex $u$ in $A_{(k-3)/2}$, $u$ cannot be adjacent to $v_1$, $v_{k-1}$, or consecutive vertices on $P_{k-1}$, $N_{P_{k-1}}(u)=\{v_2,v_4,\cdots,v_{k-3}\}$ or $N_{P_{k-1}}(u)=\{v_2,v_4,\cdots,v_{2l}, v_{2l+3},$ $v_{2l+5},\cdots,v_{k-2}\}$ where $1\leq l\leq \frac{k-5}{2}$. Let $G_1=G\setminus V(P_{k-1})$. Then
      \begin{align}
         A_{(k-3)/2}=\cup_{l=1}^{(k-5)/2}N_{G_1}^{co}\big(\big\{v_2,v_4,
         \cdots,v_{2l},v_{2l+3},v_{2l+5},\cdots,v_{k-2}\big\}\big)\nonumber\cup N_{G_1}^{co}\big(\{v_2,v_4,\cdots,v_{k-3}\}\big).\nonumber
      \end{align}
      Since $\big|A_{(k-3)/2}\big|=\Theta(n)$, we must have $|N_{G_1}^{co}(\{v_2,v_4,\cdots,v_{k-3}\})|=\Theta(n)$ or $|N_{G_1}^{co}(\{v_2,v_4,\cdots,$  $v_{2l},v_{2l+3},v_{2l+5},\cdots,v_{k-2}\})|=\Theta(n)$ for some $l$.

      For the former case, $G[\{v_2,v_4,\cdots,v_{k-3}\}]$ is $\mathcal{H}$-free. Moreover, by Lemma \ref{lemma_2}, there are no edges in $G[\{v_1,v_3,\cdots,v_{k-2}\}]$ and $v_1v_{k-1},v_3v_{k-1},\cdots,v_{k-4}v_{k-1}$ are not edges of $G$. Then
      \begin{eqnarray*}
      e(G[V(P_{k-1})])&\leq& ex\left(\frac{k-3}{2},\mathcal{H}\right)+\left(\frac{k-1}{2}\right)^2-\frac{k-3}{2}+\frac{k-3}{2}\\[2mm]
      &=&ex\left(\frac{k-3}{2},\mathcal{H}\right)+\left(\frac{k-1}{2}\right)^2\\[2mm]
      &=&ex\left(\frac{k-3}{2},\mathcal{H}\right)+\frac{k-3}{2}\frac{k+1}{2}+1.
      \end{eqnarray*}

      For the latter case, $G[\{v_2,v_4,\cdots,v_{2l},v_{2l+3},v_{2l+5},\cdots,v_{k-2}\}]$ is $\mathcal{H}$-free. Moreover, by Lemma \ref{lemma_2}, there are no edges in $G[\{v_1,v_3,\cdots,v_{2l+1},v_{2l+4},v_{2l+6},\cdots,v_{k-1}\}]$ and $v_1v_{2l+2},v_3v_{2l+2}, \cdots,$ $v_{2l-1}v_{2l+2},v_{2l+2}v_{2l+4},v_{2l+2}v_{2l+6},\cdots,v_{2l+2}v_{k-1}$ are not edges of $G$. Then
       \begin{eqnarray*}
      e(G[V(P_{k-1})])&\leq & ex\left(\frac{k-3}{2},\mathcal{H}\right)+\left(\frac{k-1}{2}\right)^2-\frac{k-3}{2}+\frac{k-3}{2}\\[2mm]
      &=& ex\left(\frac{k-3}{2},\mathcal{H}\right)+\left(\frac{k-1}{2}\right)^2\\[2mm]
      &=& ex\left(\frac{k-3}{2},\mathcal{H}\right)+\frac{k-3}{2}\frac{k+1}{2}+1.
   \end{eqnarray*} \end{proof}

   \begin{lemma}\label{lemma_5}
      Suppose $V(G)=M\cup \{u, v\}$, where $|M|=l$,  $\mathcal{H}'$ is the family of graphs obtained from $H$ by deleting two adjacent vertices in $V(H)$. If $G[M]$ is $\mathcal{H}$-free, $G$ is $H$-free and every graph in $EX(l,\mathcal{H})$ contains at least a member in $\mathcal{H}'$. Then $e(G)\leq ex(l,\mathcal{H})+2l$.
   \end{lemma}
   \begin{proof}
       Since $G[M]$ is $\mathcal{H}$-free, $e(G[M])\leq ex(l,\mathcal{H})$. If $e(G)=ex(l,\mathcal{H})+2l+1$, then $G\cong T\bigvee K_2$ where $T$ is a graph in $EX(l,\mathcal{H})$. However, since every graph in $EX(l,\mathcal{H})$ contains at least a member in $\mathcal{H}'$, $G$ must contains  $H$ as a subgraph, a contradiction.
   \end{proof}

   \begin{lemma}\label{lemma_6}
       For odd $k$, let $G$ be the graph in Lemma \ref{lemma_4} and $\big|A_{(k-3)/2}\big|=\Theta(n)$, $\mathcal{H}'$ be the family of graphs obtained from $H$ by deleting two adjacent vertices in $V(H)$. If every graph in $EX(\frac{k-3}{2},\mathcal{H})$ contains at least a member in $\mathcal{H}'$, then
\begin{equation*}
e(G[V(P_{k-1})])\leq ex\left(\frac{k-3}{2},\mathcal{H}\right)+\frac{k-3}{2}\frac{k+1}{2}
\end{equation*}
for sufficiently large $n$.
   \end{lemma}
   \begin{proof}
As in the proof of Lemma \ref{lemma_4}, we have  \begin{align}
A_{(k-3)/2}=\cup_{l=1}^{(k-5)/2}N_{G_1}^{co}\big(\{v_2,v_4,\cdots,v_{2l},v_{2l+3},v_{2l+5},
         \cdots, v_{k-2}\}\big)\nonumber\cup N_{G_1}^{co}\big(\{v_2,v_4,\cdots,v_{k-3}\}\big).\nonumber
      \end{align}
The assumption $|A_{(k-3)/2}|=\Theta(n)$ implies that $|N_{G_1}^{co}(\{v_2,v_4,\cdots,v_{k-3}\})|=\Theta(n)$ or
there exists an $l$ with $1\leq l\leq \frac{k-5}{2}$ such that $|N_{G_1}^{co}(\{v_2,v_4,\cdots,v_{2l},v_{2l+3},v_{2l+5},\cdots,v_{k-2}\})|=\Theta(n)$.

       If $|N_{G_1}^{co}(\{v_2,v_4,\cdots,v_{k-3}\})|=\Theta(n)$, then $G[\{v_2,v_4,\cdots,v_{k-3}\}]$ is $\mathcal{H}$-free,  and $G[\{v_2,v_4,\\\cdots,v_{k-3},v_{k-2},v_{k-1}\}]$ is $H$-free. By Lemma \ref{lemma_5},
\begin{equation*}
e\big(G[\{v_2,v_4,\cdots,v_{k-3},v_{k-2},v_{k-1}\}]\big)\leq ex\Big(\frac{k-3}{2},\mathcal{H}\Big)+k-3.
\end{equation*}
        Moreover, by Lemma \ref{lemma_2}, there are no edges in $G[\{v_1,v_3,\cdots,v_{k-2}\}]$ and $v_1v_{k-1},v_3v_{k-1},\cdots,\\v_{k-4}v_{k-1}$  are not edges of $G$. Then
        \begin{eqnarray*}
        e(G[V(P_{k-1})])&\leq &ex\left(\frac{k-3}{2},\mathcal{H}\right)+\left(\frac{k-3}{2}\right)^2+k-3\\[2mm]
        &=& ex\left(\frac{k-3}{2},\mathcal{H}\right)+\frac{k-3}{2}\frac{k+1}{2}.
        \end{eqnarray*}

       If $|N_{G_1}^{co}\big(\{v_2,v_4,\cdots,v_{2l},v_{2l+3},v_{2l+5},\cdots,v_{k-2}\}\big)|=\Theta(n)$ for some $l$, then $G[\{v_2,v_4,\cdots,$ $v_{2l},$ $v_{2l+3},v_{2l+5},\cdots,v_{k-2}\}]$ is $\mathcal{H}$-free,   and $G[\{v_2,v_4,\cdots,v_{2l},v_{2l+3},v_{2l+5},\cdots,v_{k-2},v_{2l+1},v_{2l+2}\}]$ is $H$-free. By Lemma \ref{lemma_5},
 \begin{equation*}
 e\big(G[\{v_2,v_4,\cdots,v_{2l},v_{2l+3},v_{2l+5},\cdots,v_{k-2},v_{2l+1},v_{2l+2}\}]\big)\leq ex\left(\frac{k-3}{2},\mathcal{H}\right)+k-3.
 \end{equation*}
 Moreover, by Lemma \ref{lemma_2}, there are no edges in $G[\{v_1,v_3,\cdots,v_{2l+1},v_{2l+4},v_{2l+6},\cdots,v_{k-1}\}]$ and $v_1v_{2l+2},v_3v_{2l+2},\cdots,v_{2l-1}v_{2l+2},v_{2l+2}v_{2l+4},v_{2l+2}v_{2l+6},\cdots,v_{2l+2}v_{k-1}$  are not edges of $G$. Then
       \begin{eqnarray*}
       e(G[V(P_{k-1})])&\leq &ex\left(\frac{k-3}{2},\mathcal{H}\right)+\left(\frac{k-3}{2}\right)^2+k-3\\[2mm]
       &=&ex\left(\frac{k-3}{2},\mathcal{H}\right)+\frac{k-3}{2}\frac{k+1}{2}.
       \end{eqnarray*}
 \end{proof}
 Let $M_t$ be a matching on $t$ edges and $S_t$ be a star on $t+1$ vertices. In 1972, Abbott, Hanson and Sauer\cite{1972-Abbott-Intersection} determined $ex\big(n,\{M_t,S_t\}\big)$.
   \begin{theorem}[\cite{1972-Abbott-Intersection}]\label{lemma_7}
      We omit isolated vertices in extremal graphs. If $k$ is odd and $n\geq 2t$, $ex(n,\{M_t,S_t\})=t^2-t$ and the extremal graph is $K_t \cup K_t$. When $k$ is even and $n\geq 2t-1$, $ex(n,\{M_t,S_t\})=t^2-\frac{3}{2}t$ and the extremal graphs are all the graphs with $2t-1$ vertices, $t^2-\frac{3}{2}t$ edges and maximum degree $t-1$.
   \end{theorem}

   We also need the stronger version proved by Chv\'{a}tal and Hanson \cite{1976-Chvatal-degrees}.
   \begin{theorem}[\cite{1976-Chvatal-degrees}]
      For every $k\geq 1$ and $t\geq 1$,
 \begin{equation*}
 ex(n,\{M_{k+1},S_{t+1}\})=k t+\Big\lfloor\frac{t}{2}\Big\rfloor\cdot\Big\lfloor\frac{k}{\lceil t/2\rceil}\Big\rfloor\leq kt+k.
 \end{equation*}
   \end{theorem}

   \section{Proofs of Theorems \ref{theorem_1},  \ref{theorem_2} and \ref{theorem_4}}

   \noindent
   {\bf Proof of Theorems \ref{theorem_1}.}
   For any graph $T$ in $EX(\frac{k}{2}-1,\mathcal{H})$, obviously $T\vee I_{n-k/2+1}$ is a $\{P_k,H\}$-free connected graph. Then
   \begin{eqnarray}\label{eq1}
   ex_{conn}(n,\{P_k,H\})\geq ex\Big(\frac{k}{2}-1,\mathcal{H}\Big)+\Big(\frac{k}{2}-1\Big)\Big(n-\frac{k}{2}+1\Big).
   \end{eqnarray}

   Next we show that $ex_{conn}(n,\{P_k,H\})\leq ex(\frac{k}{2}-1,\mathcal{H})+(\frac{k}{2}-1)(n-\frac{k}{2}+1)$. Let $G \in EX_{conn}(n, \{P_k,H\})$.

   If $P_{k-1}\nsubseteq G$, by Theorem \ref{th_1},
   \begin{eqnarray*}
   e(G)\leq \frac{k-3}{2}n < ex\Big(\frac{k}{2}-1,\mathcal{H}\Big)+\Big(\frac{k}{2}-1\Big)\Big(n-\frac{k}{2}+1\Big)
   \end{eqnarray*}
   for sufficiently large $n$.

   If $P_{k-1}\subseteq G$, we claim that $|A_{k/2-1}|=\Theta(n)$. Otherwise, $|A_{k/2-1}|=o(n)$.  By Lemma \ref{lemma_3},
   \begin{align}
      e(G)-e(G[V(P_{k-1})])&\leq \Big(\frac{k}{2}-1\Big)\Big|A_{k/2-1}\Big|+\sum_{l=0}^{k/2-2}\Big(\frac{k}{2}-\frac{3}{2}\Big)|A_l|\nonumber\\[2mm]
      &=\Big(\frac{k}{2}-1\Big)\Big|A_{k/2-1}\Big|+\Big(\frac{k}{2}-\frac{3}{2}\Big)\Big(n-k+1-
      \Big|A_{k/2-1}\Big|\Big)\nonumber\\[2mm]
      &=\Big(\frac{k}{2}-\frac{3}{2}\Big)n+o(n).\nonumber
   \end{align}
  Combining with
  $e(G[V(P_{k-1})])\leq ex(k-1,H)$, we have
   \begin{eqnarray*}
    e(G)&\leq& ex(k-1,H)+\Big(\frac{k}{2}-\frac{3}{2}\Big)n+o(n)\\[2mm]
    &<&ex\Big(\frac{k}{2}-1,\mathcal{H}\Big)+\Big(\frac{k}{2}-1\Big)\Big(n-\frac{k}{2}+1\Big)
     \end{eqnarray*}
     for sufficiently large $n$, which is a contradiction to (\ref{eq1}). So $|A_{k/2-1}|=\Theta(n)$.
     By    Lemma \ref{lemma_4},
  \begin{equation*}
  e(G[V(P_{k-1})])\leq ex\Big(\frac{k}{2}-1,\mathcal{H}\Big)+\Big(\frac{k}{2}-1\Big)\frac{k}{2}.
  \end{equation*}
        According to    Lemma \ref{lemma_3},
   \begin{eqnarray*}
   e(G)-e(G[V(P_{k-1})])&\leq &\Big(\frac{k}{2}-1\Big)\Big|A_{k/2-1}\Big|+\sum_{l=0}^{k/2-2}\Big(\frac{k}{2}-\frac{3}{2}\Big)|A_l|\\[2mm]
   &\leq &\Big(\frac{k}{2}-1\Big)(n-k+1).
   \end{eqnarray*}
     Then
   \begin{eqnarray*}
   e(G)&\leq &ex\Big(\frac{k}{2}-1,\mathcal{H}\Big)+\Big(\frac{k}{2}-1\Big)\frac{k}{2}+\Big(\frac{k}{2}-1\Big)(n-k+1)\\[2mm]
   &= &ex\Big(\frac{k}{2}-1,\mathcal{H}\Big)+\Big(\frac{k}{2}-1\Big)\Big(n-\frac{k}{2}+1\Big).
\end{eqnarray*}
So $ex_{conn}\big(n,\{P_k,H\}\big)=ex\big(\frac{k}{2}-1,\mathcal{H}\big)+\big(\frac{k}{2}-1\big)\big(n-\frac{k}{2}+1\big)$.
The extremal graph is $T\vee I_{n-k/2+1}$, where $T\in EX(\frac{k}{2}-1,\mathcal{H})$.
 \qed

\noindent
{\bf Proof of Theorem \ref{theorem_2}.}
   For any graph $T$ in $EX(\frac{k-3}{2},\mathcal{H})$, obviously $T\vee I_{n-(k-3)/2}$ is a $\{P_k,H\}$-free connected graph. Then
   \begin{eqnarray}\label{eq2}
   ex_{conn}(n,\{P_k,H\})\geq ex\Big(\frac{k-3}{2},\mathcal{H}\Big)+\frac{k-3}{2}\Big(n-\frac{k-3}{2}\Big).
   \end{eqnarray}

   Next we will show that $ex_{conn}\big(n,\{P_k,H\}\big)\leq ex\big(\frac{k-3}{2},\mathcal{H}\big)+\frac{k-3}{2}(n-\frac{k-3}{2})+1.$ Let $G \in EX_{conn}\big(n, \{P_k,H\}\big)$.
   If $P_{k-2}\nsubseteq G$, by Theorem \ref{th_1}, $e(G)\leq \frac{k-4}{2}n<ex\big(\frac{k-3}{2},\mathcal{H}\big)+\frac{k-3}{2}\big(n-\frac{k-3}{2}\big)$ for sufficiently large $n$.

   If $P_{k-1}\nsubseteq G$, $P_{k-2}\subseteq G$, we claim that   $\big|A_{(k-3)/2}\big|=\Theta(n)$. Otherwise, $\big|A_{(k-3)/2}\big|=o(n)$.
    By Lemma \ref{lemma_3},
   \begin{align}
      e(G)-e(G[V(P_{k-2})])&\leq \frac{k-3}{2}\big|A_{(k-3)/2}\big|+\sum_{l=0}^{(k-5)/2}\frac{k-4}{2}|A_l|\nonumber\\[2mm]
      &=\frac{k-3}{2}\big|A_{(k-3)/2}\big|+\frac{k-4}{2}\big(n-k+2-\big|A_{(k-3)/2}\big|\big)\nonumber\\[2mm]
      &=\frac{k-4}{2}n+o(n).\nonumber
   \end{align}
  Combining with $e(G[V(P_{k-2})])\leq ex(k-2,H)$, we have
  \begin{eqnarray*}
  e(G)&\leq& ex(k-2,H)+\frac{k-4}{2}n+o(n)\\[2mm]
  &<& ex\Big(\frac{k-3}{2},\mathcal{H}\Big)+\frac{k-3}{2}\Big(n-\frac{k-3}{2}\Big)
  \end{eqnarray*}
  for sufficiently large $n$, which is a contradiction to (\ref{eq2}). So $\big|A_{(k-3)/2}\big|=\Theta(n)$.

   By Lemmas  \ref{lemma_4} and \ref{lemma_3}, we have
   \begin{eqnarray*}
   e(G[V(P_{k-2})])\leq ex\Big(\frac{k-3}{2},\mathcal{H}\Big)+\frac{k-3}{2}\frac{k-1}{2},
   \end{eqnarray*}
   and
   \begin{eqnarray*}
   e(G)-e(G[V(P_{k-2})])&\leq &\frac{k-3}{2}\big|A_{(k-3)/2}\big|+\sum_{l=0}^{(k-5)/2}\frac{k-4}{2}|A_l|\\[2mm]
  & \leq &\frac{k-3}{2}(n-k+2).
 \end{eqnarray*}
 Then
 \begin{eqnarray*}
 e(G)&\leq &ex\Big(\frac{k-3}{2},\mathcal{H}\Big)+\frac{k-3}{2}\frac{k-1}{2}+\frac{k-3}{2}(n-k+2)\\[2mm]
 &=& ex\Big(\frac{k-3}{2},\mathcal{H}\Big)+\frac{k-3}{2}\Big(n-\frac{k-3}{2}\Big).
 \end{eqnarray*}

   If $P_{k-1}\subseteq G$, we claim that   $|A_{(k-3)/2}|=\Theta(n)$. Otherwise, $|A_{(k-3)/2}|=o(n)$. By Lemma \ref{lemma_3},
   \begin{align}
      e(G)-e(G[V(P_{k-1})])&\leq \frac{k-3}{2}\big|A_{(k-3)/2}\big|+\sum_{l=0}^{(k-5)/2}\frac{k-4}{2}|A_l|\nonumber\\[2mm]
      &=\frac{k-3}{2}\big|A_{(k-3)/2}\big|+\frac{k-4}{2}\big(n-k+1-\big|A_{(k-3)/2}\big|\big)\nonumber\\[2mm]
      &=\frac{k-4}{2}n+o(n).\nonumber
   \end{align}
  Combining with  $e(G[V(P_{k-1})])\leq ex(k-1,H)$, we have
   \begin{eqnarray*}
  e(G)&\leq& ex(k-1,H)+\frac{k-4}{2}n+o(n)\\[2mm]
  &<& ex\Big(\frac{k-3}{2},\mathcal{H}\Big)+\frac{k-3}{2}\Big(n-\frac{k-3}{2}\Big),
   \end{eqnarray*}
   for sufficiently large $n$, which is a contradiction to (\ref{eq2}). So $\big|A_{(k-3)/2}\big|=\Theta(n)$.
   By Lemmas \ref{lemma_4} and \ref{lemma_3}, we have
    \begin{eqnarray*}
   e\big(G[V(P_{k-1})]\big)\leq ex\Big(\frac{k-3}{2},\mathcal{H}\Big)+\frac{k-3}{2}\frac{k+1}{2}+1,
    \end{eqnarray*}
   and
    \begin{eqnarray*}
   e(G)-e(G[V(P_{k-1})])&\leq & \frac{k-3}{2}\big|A_{(k-3)/2}\big|+\sum_{l=0}^{(k-5)/2}\frac{k-4}{2}|A_l|\\
   &\leq& \frac{k-3}{2}(n-k+1).
   \end{eqnarray*}
   Then
   \begin{eqnarray*}
   e(G)&\leq &ex\Big(\frac{k-3}{2},\mathcal{H}\Big)+\frac{k-3}{2}\frac{k+1}{2}+1+\frac{k-3}{2}(n-k+1)\\[2mm]
   &=& ex\Big(\frac{k-3}{2},\mathcal{H}\Big)+\frac{k-3}{2}\Big(n-\frac{k-3}{2}\Big)+1.
\end{eqnarray*}
Consequently, $ex_{conn}\big(n,\{P_k,H\}\big)=ex\big(\frac{k-3}{2},\mathcal{H}\big)+\frac{k-3}{2}\big(n-\frac{k-3}{2}\big)+c$, where $c=0$ or $1$.
\qed

\noindent
{\bf Proof of Theorem \ref{theorem_4}.} Let $G \in EX_{conn}(n, \{P_k,H\})$. According to the proof of Theorem \ref{theorem_2}, we only need to consider the case $P_{k-1}\subseteq G$. As in the proof of Theorem \ref{theorem_2},  we can show that $\big|A_{(k-3)/2}\big|=\Theta(n)$.
   By Lemma \ref{lemma_6}, $$e(G[V(P_{k-1})])\leq ex\Big(\frac{k-3}{2},\mathcal{H}\Big)+\frac{k-3}{2}\frac{k+1}{2}.$$ Then
   \begin{eqnarray*}
   e(G)&\leq & ex\Big(\frac{k-3}{2},\mathcal{H}\Big)+\frac{k-3}{2}\frac{k+1}{2}+\frac{k-3}{2}(n-k+1)\\[2mm]
   &= &ex\Big(\frac{k-3}{2},\mathcal{H}\Big)+\frac{k-3}{2}\Big(n-\frac{k-3}{2}\Big).
   \end{eqnarray*}
   So
   \begin{eqnarray*}
   ex_{conn}(n,\{P_k,H\})\leq ex\Big(\frac{k-3}{2},\mathcal{H}\Big)+\frac{k-3}{2}\Big(n-\frac{k-3}{2}\Big).
    \end{eqnarray*}
    Combining with the lower bound
    \begin{eqnarray*}
    ex_{conn}(n,\{P_k,H\})\geq ex\Big(\frac{k-3}{2},\mathcal{H}\Big)+\frac{k-3}{2}\Big(n-\frac{k-3}{2}\Big),
    \end{eqnarray*}
     we have $$ex_{conn}(n,\{P_k,H\})= ex\Big(\frac{k-3}{2},\mathcal{H}\Big)+\frac{k-3}{2}\Big(n-\frac{k-3}{2}\Big).$$
    Moreover, $T\vee I_{n-(k-3)/2}$ is an extremal graph, where $T$ is a graph in $EX\Big(\frac{k-3}{2},\mathcal{H}\Big)$. \qed

   We show that the above results imply Theorem \ref{th_3}. When $k$ is even, Theorem \ref{theorem_1} yields that
    \begin{eqnarray*}
   ex_{conn}(n,\{P_k,K_m\})&=& ex\Big(\frac{k}{2}-1,K_{m-1}\Big)+\Big(\frac{k}{2}-1\Big)\Big(n-\frac{k}{2}+1\Big)\\[2mm]
   &=& e\Big(T\Big(\frac{k}{2}-1,m-2\Big)\Big)+\Big(\frac{k}{2}-1\Big)\Big(n-\frac{k}{2}+1\Big),
    \end{eqnarray*}
   and the extremal graph is $T\big(\frac{k}{2}-1,m-2\big)\vee I_{n-k/2+1}$. When $k$ is odd, we have the following corollary.

\begin{corollary}
    For $k$ is odd, if $k\geq 2m-1$, then $ex_{conn}(n,\{K_m,P_k\})=e\big(T(\frac{k-3}{2},m-2)\big)+\frac{k-3}{2}\Big(n-\frac{k-3}{2}\Big)$ for  sufficiently large $n$, and the extremal graph is $T(\frac{k-3}{2},m-2)\vee I_{n-(k-3)/2}$. If $m<k<2m-1$, $ex_{conn}(n,\{K_m,P_k\})=\frac{(k-3)(k-5)}{8}+\frac{k-3}{2}\Big(n-\frac{k-3}{2}\Big)+1$ and the extremal graph is $K_{(k-3)/2}\vee \big(I_{n-(k-1)/2}\cup K_2\big)$.
\end{corollary}

\begin{proof}
    Take $H=K_m$, then $\mathcal{H}=\{K_{m-1}\}$, $EX\big(\frac{k-3}{2},K_{m-1}\big)=\big\{T\big(\frac{k-3}{2},m-2\big)\big\}$ and $\mathcal{H}'=\{K_{m-2}\}$. If $k\geq 2m-1$, since $T(\frac{k-3}{2},m-2)$ contains a $K_{m-2}$, by Theorem \ref{theorem_4}, $ex_{conn}(n,\{P_k,K_m\})=e\big(T(\frac{k-3}{2},m-2)\big)+\frac{k-3}{2}\big(n-\frac{k-3}{2}\big)$ and the extremal graph is $T(\frac{k-3}{2},m-2)\vee I_{n-(k-3)/2}$.

    If $m<k<2m-1$, it is easily  seen that  $K_{(k-3)/2}\vee (I_{n-(k-1)/2}\cup K_2)$ is $\{P_k,K_m\}$-free, by Theorem \ref{theorem_2}, we have
   \begin{eqnarray*}
   ex_{conn}(n,\{K_m,P_k\})&=&e\Big(T\Big(\frac{k-3}{2},m-2\Big)\Big)+\frac{k-3}{2}\Big(n-\frac{k-3}{2}\Big)+1\\[2mm]
   &=&\frac{(k-3)(k-5)}{8}+\frac{k-3}{2}\Big(n-\frac{k-3}{2}\Big)+1.
   \end{eqnarray*} \end{proof}

 Take $H=F_t$,  where $F_t$ is the graph formed by $t$ triangles intersecting at a vertex. Then $\mathcal{H}=\{M_t,S_t\}$ and $\mathcal{H}'=\{F_{t-1},M_{t-1}\cup K_1\}$. By Theorem \ref{lemma_7}, when $k$ is odd, the extremal graph for  $\{M_t,S_t\}$ is
 $K_t \cup K_t$. When $k$ is even, the extremal graphs for  $\{M_t,S_t\}$
 are all the graphs with $2t-1$ vertices, $t^2-\frac{3}{2}t$ edges and maximum degree $t-1$.
\begin{corollary}
   For sufficiently large $n$, $ex_{conn}(n,\{F_t,P_k\})=ex\big(\lfloor\frac{k}{2}\rfloor-1,\{M_t,S_t\}\big)+\big(\lfloor\frac{k}{2}\rfloor-1\big)
   \big(n-\lfloor\frac{k}{2}\rfloor+1\big)$ and the extremal graph is $T\vee I_{n-\lfloor k/2\rfloor+1}$, where $T\in EX\big(\lfloor\frac{k}{2}\rfloor-1,\{M_t,S_t\}\big)$.
\end{corollary}

\begin{proof} Take $H=F_t$.
    When $k$ is even, Theorem \ref{theorem_1} yields
   \begin{eqnarray*}
ex_{conn}(n,\{P_k,F_t\})=ex\Big(\frac{k}{2}-1,\{M_t,S_t\}\Big)+\Big(\frac{k}{2}-1\Big)\Big(n-\frac{k}{2}+1\Big),
    \end{eqnarray*}
     and the extremal graph is $T\vee I_{n-k/2+1}$, where $T\in EX(\frac{k}{2}-1,\{M_t,S_t\})$.

     If $k$ is odd, then $\mathcal{H}=\{M_t,S_t\}$ and $\mathcal{H}'=\{F_{t-1},M_{t-1}\cup K_1\}$. We first show that every graph in $EX(\frac{k-3}{2},\{M_t,S_t\})$ contains a copy of $M_{t-1}\cup K_1$.
     By Theorem \ref{lemma_6}, $ex\big(\frac{k-3}{2},\{M_t,S_t\}\big)=t^2-t$ if $t$ is odd and $ex\big(\frac{k-3}{2},\{M_t,S_t\}\big)=t^2-\frac{3}{2}t$ otherwise. By Theorem \ref{lemma_7}, $ex\big(\frac{k-3}{2},\{M_{t-1},S_t\}\big)\leq (t-1)^2$. Then $ex\big(\frac{k-3}{2},\{M_t,S_t\}\big)>ex\big(\frac{k-3}{2},\{M_{t-1},S_t\}\big)$. So  every graph in $EX(\frac{k-3}{2},\{M_t,S_t\})$ must contain  a copy of  $M_{t-1}\cup K_1$. Hence, by Theorem \ref{theorem_4},
    $$ex_{conn}(n,\{P_k,F_t\})=ex\Big(\frac{k-3}{2},\{M_t,S_t\}\Big)+\frac{k-3}{2}\Big(n-\frac{k-3}{2}\Big),$$ and the extremal graph is $T\vee I_{n-(k-3)/2}$, where $T\in EX\big(\frac{k-3}{2},\{M_t,S_t\}\big)$.
\end{proof}
{\bf Remark.}
   We  give an example for  $c=1$ in Theorem \ref{theorem_2}.
   Take $H=K_{2,2,2}$. We will show that $ex_{conn}(n,\{P_k,H\})=ex\big(\frac{k-3}{2},\mathcal{H}\big)+\frac{k-3}{2}\big(n-\frac{k-3}{2}\big)+1$.

   For $H=K_{2,2,2}$, then $\mathcal{H}=C_4$. It is sufficient to prove that for every $T\in EX(\frac{k-3}{2},C_4)$, $G=T\vee \big(I_{n-(k-1)/2}\cup \{uv\}\big)$ is $\{P_k,K_{2,2,2}\}$-free. Obviously $G$ is $P_k$-free.
     If $G$ contains a copy  of $K_{2,2,2}$, say $W$, then $W$ must contain $uv$. Note that every edge of $K_{2,2,2}$ belongs to a $C_4$ in $K_{2,2,2}$, then $uv$ belongs to a $C_4$ in $W$. Let the $C_4$  containing $uv$ in $W$ be $\{x,y,u,v\}$, then $x,y\in V(T)$. Suppose $\{a,b\}=V(W)\setminus \{x,y,u,v\}$, then vertices $a,b$ must be in  $T$. However, $W[\{a,b,x,y\}]$ is a $C_4$ in $T$, which contradicts to that $T$ is $C_4$-free. Then by Theorem \ref{theorem_2}, we have  $$ex_{conn}(n,\{P_k,K_{2,2,2}\})=ex\Big(\frac{k-3}{2},C_4\Big)+\frac{k-3}{2}\Big(n-\frac{k-3}{2}\Big)+1.$$

\section{Proof of Theorem \ref{theorem_3}}

For the lower bound,  it is easily seen that  $ex(n,\{P_k,H\})\geq ex_{conn}(n,\{P_k,H\})=(\lfloor\frac{k}{2}\rfloor-1)n+O_k(1)$ by Theorems \ref{theorem_1} and \ref{theorem_2}.  Moreover, suppose $n=m(k-1)+r$, where $0\leq r\leq k-2$. Then the disjoint union of $m$ copies from  a graph in $EX(k-1,H)$ and a copy from a graph in  $EX(r,H)$ is also $\{P_k,H\}$-free. Hence
 \begin{eqnarray*}
 ex\big(n,\{P_k,H\}\big)&\geq &m\cdot ex(k-1,H)+ex(r,H)\\
 &=&\frac{ex(k-1,H)}{k-1}n+O_k(1).
 \end{eqnarray*}
 So $ex(n,\{H,P_k\})\geq n\max\big\{\lfloor \frac{k}{2}\rfloor-1,\frac{ex(k-1,H)}{k-1}\big\}+O_k(1)$.

Suppose $G$ is an $n$-vertex $\{P_k,H\}$-free graph. Let $G_i\,(1\leq i\leq t)$ be the components of $G$ and $l_i=|V(G_i)|$.

{\em Case 1.} $k$ is even.
 For the upper bound, we distinguish two cases for $l_i$.

 \noindent
 {\em Case 1.1.}  $l_i<k$. Then $e(G_i)\leq ex(l_i,H)\leq \frac{l_i(l_i-1)}{2}\leq \frac{k-2}{2}l_i$.

\noindent
{\em Case 1.2.}  $l_i\geq k$. Let the longest path in $G_i$ be $P_m$, where $m\leq k-1$. Then
\begin{align}
   e(G_i)&\leq ex(m,H)+\Big(\frac{k}{2}-1\Big)(l_i-m) \nonumber\\[2mm]
   &\leq \frac{m(m-1)}{2}+\Big(\frac{k}{2}-1\Big)l_i-\frac{m(k-2)}{2}\nonumber\\[2mm]
   &=\Big(\frac{k}{2}-1\Big)l_i-\frac{m}{2}(m-k+1)\nonumber\\[2mm]
   &\leq \Big(\frac{k}{2}-1\Big)l_i\nonumber.
\end{align}
The first inequality holds by $e(G_i[V(P_m)])\leq ex(m,H)$ and Lemma \ref{lemma_3}.
Therefore,
$$e(G)=\sum_{i=1}^te(G_i)\leq \sum_{i=1}^t\Big(\frac{k}{2}-1\Big)l_i=\Big(\frac{k}{2}-1\Big)n.$$
Then $ex(n,\{H,P_k\})\leq n\cdot\max\big\{\big\lfloor \frac{k}{2}\big\rfloor-1,\frac{ex(k-1,H)}{k-1}\big\}$. Combining with the lower bound, we have
$$ex(n,\{H,P_k\})= n \cdot\max\Big\{\Big\lfloor \frac{k}{2}\Big\rfloor-1,\frac{ex(k-1,H)}{k-1}\Big\}+O_k(1).$$

 {\em Case 2.} $k$ is odd. For the upper bound, we distinguish three different cases for $l_i$.

\noindent
{\em Case 2.1.}  $l_i<k-1$. Then $e(G_i)\leq ex(l_i,H)\leq \frac{l_i(l_i-1)}{2}\leq \frac{k-3}{2}l_i$.

\noindent
{\em Case 2.2.}  $l_i=k-1$. Then $e(G_i)\leq ex(k-1,H)$.

\noindent
{\em Case 2.3.}  $l_i\geq k$. Let the longest path in $G_i$ be $P_m$ where $m\leq k-1$.

 If $P_{k-1}\nsubseteq G_i$, then $m\leq k-2$ and
\begin{align}
   e(G_i)&\leq ex(m,H)+\frac{k-3}{2}(l_i-m) \nonumber\\[2mm]
   &\leq \frac{m(m-1)}{2}+\frac{k-3}{2}l_i-\frac{m(k-3)}{2}\nonumber\\[2mm]
   &=\frac{k-3}{2}l_i-\frac{m}{2}(m-k+2)\nonumber\\[2mm]
   &\leq \frac{k-3}{2}l_i\nonumber.
\end{align}
The first inequality holds by $e(G_i[V(P_m)])\leq ex(m,H)$ and Lemma \ref{lemma_3}.

\noindent
If $P_{k-1}\subseteq G_i$, then
\begin{eqnarray*}
e(G_i)&\leq &ex(k-1,H)+\frac{k-3}{2}(l_i-k+1)\\[2mm]
&=& \frac{k-3}{2}l_i+ex(k-1,H)-\frac{(k-1)(k-3)}{2}.
\end{eqnarray*}
The first inequality holds by $e(G_i[V(P_{k-1})])\leq ex(k-1,H)$
 and Lemma \ref{lemma_3}.

Let $n_1:=\sum_{l_j<k-1} l_j$, $n_2:=\sum_{l_j=k-1}l_j$, $n_3:=\sum_{l_j\geq k,P_{k-1}\nsubseteq G_j}l_j$ and $n_4:=\sum_{l_j\geq k,P_{k-1}\subseteq G_j}l_j$. Then we have
\begin{align}
   e(G)&=\frac{k-3}{2}n_1+\frac{n_2}{k-1}ex(k-1,H)+\sum_{l_i\geq k,P_{k-1}\nsubseteq G_i}e(G_i)+\sum_{l_j\geq k,P_{k-1}\subseteq G_j}e(G_j)\nonumber\\[2mm]
   &\leq \frac{k-3}{2}n_1+\frac{n_2}{k-1}ex(k-1,H)+\frac{k-3}{2}n_3+\frac{k-3}{2}n_4+\frac{n_4}{k}
   \Big(ex(k-1,H)-\frac{(k-1)(k-3)}{2}\Big)\nonumber\\[2mm]
   &\leq \frac{k-3}{2}n_1+\frac{n_2}{k-1}ex(k-1,H)+\frac{k-3}{2}n_3+\frac{k-3}{2}n_4+
   \frac{n_4}{k-1}\Big(ex(k-1,H)-\frac{(k-1)(k-3)}{2}\Big)\nonumber\\[2mm]
   &=\frac{k-3}{2}(n_1+n_3+n_4)+\frac{n_2+n_4}{k-1}ex(k-1,H)-\frac{k-3}{2}n_4\nonumber\\[2mm]
   &=\frac{k-3}{2}(n_1+n_3)+\frac{ex(k-1,H)}{k-1}(n_2+n_4).\nonumber
\end{align}
Let $x=n_1+n_3$ and $y=n_2+n_4$. Then $x+y=n$ and $e(G)\leq \frac{k-3}{2}x+\frac{ex(k-1,H)}{k-1}y$. Hence
$$e(G)\leq n\cdot\max\Big\{\frac{k-3}{2},\frac{ex(k-1,H)}{k-1}\Big\}.$$
Then $ex(n,\{H,P_k\})\leq n\cdot\max\big\{\big\lfloor \frac{k}{2}\big\rfloor-1,\frac{ex(k-1,H)}{k-1}\big\}$. Combining with the lower bound, we have
$$ex(n,\{H,P_k\})= n \cdot\max\Big\{\Big\lfloor \frac{k}{2}\Big\rfloor-1,\frac{ex(k-1,H)}{k-1}\Big\}+O_k(1).$$
The proof is finished. \qed

\section{Concluding Remarks}
Theorem \ref{theorem_4} gives a sufficient condition for $c=0$ in Theorem \ref{theorem_2}
and we also give an example to show $c$ can be $1$ in Theorem \ref{theorem_2}. It is natural to ask whether we can determine when $c$ is $0$ or $1$ exactly in Theorem \ref{theorem_2}.

We obtain asymptotical result of $ex(n,\{P_k,H\})$ with the constant term $O_k(1)$ left,  it is an interesting question to give the exact value of $ex(n,\{P_k,H\})$.

\section*{Declaration of interests}
The authors declare that there is no conflict of interest.

\end{document}